\documentclass[11pt,final]{amsart}
\usepackage[utf8]{inputenc}
\usepackage[T1]{fontenc}
\usepackage[letterpaper,centering]{geometry}
\usepackage{lmodern}
\usepackage{eucal}
\usepackage{mathrsfs}
\usepackage{amsmath,amssymb,amsfonts,amscd}
\usepackage{microtype,xcolor,enumitem,booktabs,needspace}
\definecolor{darkblue}{rgb}{0,0,0.3}
\definecolor{darkgreen}{rgb}{0,0.4,0}
\usepackage[colorlinks=true,linkcolor=darkblue,citecolor=darkgreen,
  urlcolor=darkblue,unicode,pdfborder={0 0 0}]{hyperref}
\setlist[enumerate]{label={\upshape(\arabic*)},topsep=.7ex,leftmargin=*}
\setlist[itemize]{leftmargin=*}
\theoremstyle{plain}
\newtheorem{thm}{Theorem}[section]
\newtheorem{lem}[thm]{Lemma}
\newtheorem{prop}[thm]{Proposition}

\newtheorem*{campanaconj}{Conjecture}
\theoremstyle{definition}

\newtheorem*{terminology}{Terminology}
\newtheorem*{ack}{Acknowledgements}
\numberwithin{equation}{section}
\newcommand{\PP}{\mathbf P}
\newcommand{\Aff}{\mathbf A}
\newcommand{\A}{\mathbf A}
\newcommand{\OO}{\mathcal O}
\newcommand{\FF}{\mathbf F}
\newcommand{\ZZ}{\mathbf Z}

\newcommand{\kbar}{\overline{k}}
\newcommand{\sX}{\mathscr X}
\newcommand{\sY}{\mathscr Y}
\newcommand{\sC}{\mathscr C}
\newcommand{\sD}{\mathscr D}
\newcommand{\sE}{\mathscr E}

\newcommand{\sQ}{\mathscr Q}
\newcommand{\sR}{\mathscr R}
\newcommand{\sZ}{\mathscr Z}
\newcommand{\CC}{\mathcal C}
\newcommand{\Gm}{\mathbf G_{\mathrm m}}
\newcommand{\red}{\mathrm{red}}

\newcommand{\et}{\mathrm{\acute et}}
\DeclareMathOperator{\Br}{Br}
\DeclareMathOperator{\Pic}{Pic}
\DeclareMathOperator{\Gal}{Gal}
\DeclareMathOperator{\Aut}{Aut}
\DeclareMathOperator{\Spec}{Spec}
\DeclareMathOperator{\Supp}{Supp}
\DeclareMathOperator{\inv}{inv}
\DeclareMathOperator{\ord}{ord}
\DeclareMathOperator{\Image}{Im}
\title[Weak approximation for Campana points]
{Weak approximation for Campana points on projective spaces}

\author{Dasheng Wei}
\address{Academy of Mathematics and System Science, CAS, Beijing 100190,
P.\ R.\ China {\normalfont\itshape and} School of Mathematical Sciences,
University of CAS, Beijing 100049, P.\ R.\ China}
\email{dshwei@amss.ac.cn}

\date{September 30, 2026}

\begin{document}

\begin{abstract}
We prove weak approximation for Campana points on the projective line
with log Fano boundary, and on projective $n$-space with either $n+3$
hyperplanes in general position or two transverse smooth quadrics,
each of coefficient $1/2$.
\end{abstract}

\maketitle

\section{Introduction}\label{sec:introduction}

The paper is devoted to weak approximation for Campana points on
projective spaces over number fields. Let $X$ be a smooth projective
variety over a number field $k$, and let
\[
 D=\sum_i\left(1-\frac1{n_i}\right)D_i,\qquad 2\leq n_i<\infty,
\]
be an orbifold divisor. Fix a finite set $S$ of places containing the
archimedean places and a proper integral model of the pair. A point
$x\in(X\setminus\Supp D)(k)$ is an $S$-Campana point if
\[
 n_v(\sD_{i,\lambda},x)=0\quad\text{or}\quad
 n_v(\sD_{i,\lambda},x)\geq n_i
 \qquad(v\notin S,\ \lambda\in I_i(v)),
\]
where $\sD_i$ is the closure of $D_i$ and $\sD_{i,\lambda}$ runs
through the irreducible components of $\sD_i\times_{\OO_S}\OO_v$.
Here $n_v$ is the local intersection
multiplicity defined by the colength of the pulled-back ideal. We use
\cite[Definitions~2.3--2.4]{NS} throughout. Campana weak approximation
means that these points are dense in the product of the local Campana
spaces. These componentwise inequalities differ from the divisibility
conditions for Darmon points; see \cite[\S3.2]{PSTVA} and
\cite[\S2]{MNS}.

Campana introduced orbifolds to retain the contribution of multiple
fibres to the geometry and arithmetic of a fibration
\cite{Cam04,Cam05}. A smooth orbifold is special if, on suitable
smooth birational models, it has no orbifold fibration onto a
positive-dimensional base of general type \cite[\S4.3]{Cam11}.
The following conjecture relates this property to arithmetic.

\begin{campanaconj}[Campana, {\cite[Conjecture~12.21(1)]{Cam11}}]
Let $(X,D)$ be a smooth projective Campana orbifold over a number
field $k$, with simple normal crossings boundary and finite integral
orbifold multiplicities. Then $(X_{\kbar},D_{\kbar})$ is special if
and only if there exist a finite extension $k'/k$, a finite set
$S'\supseteq\Omega_{k'}^\infty$, and a proper integral model over
$\OO_{S'}$ with Zariski-dense $S'$-Campana points in $X_{k'}$.
\end{campanaconj}

Smooth log Fano orbifolds are special \cite[Corollary~6.9]{Cam11}.
Abramovich gives a finer formulation for a constellation supported
by a firmament, whose integral points are predicted to be potentially
dense precisely when the constellation is special
\cite[Conjecture~2.4.19]{Abr09}. We use the componentwise Campana
conditions, not those of an arbitrary firmament. The potential-density
conjecture allows a field extension and enlargement of $S$; the
approximation statements below hold over $k$ for every fixed finite set $S$.

Nakahara and Streeter prove Campana weak approximation on $\PP^n$ with
at most $n+1$ hyperplanes in general position and arbitrary finite
multiplicities, and with $n+2$ hyperplanes of multiplicity two
\cite[Theorem~1.3 and Corollary~1.4]{NS}. They also relate Campana
weak weak approximation to the Hilbert property
\cite[Theorem~1.1]{NS}. Mitankin, Nakahara and Streeter develop a
semi-integral Brauer--Manin obstruction and prove weak approximation
for a single smooth quadric boundary with arbitrary finite
multiplicity \cite[Theorem~1.3(i)]{MNS}. Related counting problems
are studied in \cite{PSTVA}. 

Our first result concerns the projective line.

\begin{thm}\label{thm:line}
Let $P_i$ be distinct closed points of $\PP^1_k$, and let
\[
 D=\sum_{i\in I}\left(1-\frac1{n_i}\right)P_i,
 \qquad 2\leq n_i<\infty.
\]
Suppose that
\begin{equation}\label{eq:log-fano-degree}
 \sum_{i\in I}\left(1-\frac1{n_i}\right)[k(P_i):k]<2.
\end{equation}
Then $(\PP^1_k,D)$ satisfies Campana weak approximation for every
fixed finite set $S\supseteq\Omega_k^\infty$ and every fixed proper
flat integral model. More precisely,
global $S$-Campana points outside $\Supp D$ are dense in the product
of the local Campana spaces.
\end{thm}

The degree inequality is equivalent to the ampleness of $2H-D$,
where $H$ is a rational point of $\PP^1$. Our next two results
concern higher-dimensional projective spaces.

\begin{thm}\label{thm:hyperplanes}
Let $n\geq2$, and let $E_1,\ldots,E_{n+3}\subset\PP^n_k$ be
distinct $k$-hyperplanes in general position. Set
\[
 D=\frac12\sum_{i=1}^{n+3}E_i,\qquad
 U=\PP^n_k\setminus\bigcup_{i=1}^{n+3}E_i.
\]
Fix a finite set $S\supseteq\Omega_k^\infty$, and let $\sE_i$ be
the closure of $E_i$ in $\PP^n_{\OO_S}$. Then the
$S$-Campana points in $U(k)$ are dense in the product of the local
Campana spaces. In particular, $(\PP^n_k,D)$ satisfies Campana weak
approximation.
\end{thm}

Here general position means that any $n+1$ defining linear forms
are linearly independent. For $n=2$ this gives five lines with no
triple intersection.

\begin{thm}\label{thm:quadrics}
Let $n\geq2$, and let $C_1,C_2\subset\PP^n_k$ be smooth
$k$-quadric hypersurfaces meeting transversely. Set
\[
 D=\frac12(C_1+C_2),\qquad
 U=\PP^n_k\setminus(C_1\cup C_2).
\]
For  any finite set $S\supseteq\Omega_k^\infty$,  the
$S$-Campana points in $U(k)$ are dense in the product of the local
Campana spaces. In particular, $(\PP^n_k,D)$ satisfies Campana weak
approximation.
\end{thm}

Transversality means that $C_1\cap C_2$ is smooth of codimension
two; neither the quadrics nor their intersection need have rational
points. The pairs in Theorems~\ref{thm:hyperplanes} and
\ref{thm:quadrics} are log Fano, since $-(K_{\PP^n}+D)$ is
respectively equivalent to $\frac{n-1}{2}H$ and $(n-1)H$.

The proofs first choose a rational point on the base, then a finite
cover with a rational point above it. In the higher-dimensional cases,
strong approximation away from a codimension-two subset ensures that
the chosen point reduces onto at most one boundary component outside
the prescribed approximation places. For Theorem~\ref{thm:line}, spherical uniformization gives a
rational curve. For the other two theorems, the covers are smooth
intersections of two quadrics. In dimension at least three, they
satisfy weak approximation. For the surface cases, we either arrange
that the Brauer group is constant or preserve Brauer evaluations
under local modifications. The essential local arguments treat every
new bad place of the cover.
Section~\ref{sec:preliminaries} records the preliminaries, and
Sections~\ref{sec:line-proof}--\ref{sec:quadric-proof} prove the
three theorems in order.

\begin{terminology}
Write $\OO_S$ for the ring of $S$-integers in $k$; the same notation
with $S'$ refers to the corresponding ring in the specified extension
field. Write $\Omega_k$ for the places of $k$, $\Omega_k^\infty$ for the
archimedean places, $k_v$ for the completion, and $\A_k$ for the
adeles. For finite $v$, write $\OO_v$, $\FF_v$, $\varpi_v$ and
$\ord_v$ for its valuation ring, residue field, a uniformizer and
the normalized valuation. We use
$\Br(Y)=H^2_{\et}(Y,\Gm)$ and
$\Br_0(Y)=\Image(\Br(k)\to\Br(Y))$. For proper $Y$, the notation
$Y(\A_k)^{\Br(Y)}$ denotes the adelic points orthogonal to
$\Br(Y)$ under the Brauer--Manin pairing.
\end{terminology}

\begin{ack}
	The work is supported by National Natural Science Foundation of China (Grant Nos. 12371014 and 12231009).  
\end{ack}

\section{Preliminaries}\label{sec:preliminaries}

Let $X$ be smooth, projective and geometrically integral, and let
$D=\sum_i(1-1/n_i)D_i$ with distinct prime divisors $D_i$ defined
over $k$ and integers $2\leq n_i<\infty$. The $D_i$ need not be
geometrically irreducible. Put $X^\circ=X\setminus\Supp D$, and
fix a proper flat integral model $\sX$ over $\OO_S$, with boundary
closures $\sD_i$.

For $v\notin S$, write
\[
 D_i\times_k k_v=\bigcup_{\lambda\in I_i(v)}D_{i,\lambda},
\]
where the $D_{i,\lambda}$ are the $k_v$-irreducible components, and
let $\sD_{i,\lambda}$ be their reduced closures in $\sX_{\OO_v}$.
These are the irreducible components of $\sD_i\times_{\OO_S}\OO_v$
used in \cite[Definition~2.2]{NS}; they are not the irreducible
components of the special fibre.

A point $x\in X(k_v)$ extends uniquely to a section
$\widetilde x:\Spec\OO_v\to\sX_{\OO_v}$. For a closed subscheme
$\mathscr E\subset\sX_{\OO_v}$ with ideal sheaf $\mathcal I_{\mathscr E}$,
put $J_{\mathscr E,x}=\widetilde x^{-1}\mathcal I_{\mathscr E}\cdot\OO_v$
and define its \emph{local intersection multiplicity} with $x$ by
\[
 n_v(\mathscr E,x)=
 \begin{cases}
 \operatorname{length}_{\OO_v}(\OO_v/J_{\mathscr E,x}),
     &J_{\mathscr E,x}\ne0,\\
 \infty,&J_{\mathscr E,x}=0.
 \end{cases}
\]
Equivalently, $J_{\mathscr E,x}=\mathfrak m_v^{\,n_v(\mathscr E,x)}$
when the intersection multiplicity is finite. If $\mathscr E$ is locally
cut out by one equation $f$ near the reduction of $x$, then
$n_v(\mathscr E,x)=\ord_v(f(x))$, with $\ord_v(0)=\infty$.
Applied to each $\sD_{i,\lambda}$, this is the intersection multiplicity
of \cite[Definition~2.3]{NS}.

Following \cite[Definition~2.4]{NS}, set
\begin{align*}
 \CC_v(X,D)&=\{x\in X(k_v):
 n_v(\sD_{i,\lambda},x)\in
       \{0\}\cup\ZZ_{\geq n_i}\cup\{\infty\}
 \text{ for all }i,\lambda\in I_i(v)\}, &&v\notin S,\\
 \CC_v(X,D)&=X(k_v),&&v\in S,\\
 \CC_S(X,D)&=\{x\in X^\circ(k):
             x\in\CC_v(X,D)\text{ for every }v\notin S\}.
\end{align*}
Thus the inequalities are imposed separately on the local
irreducible components, not merely on the total intersection multiplicity
with a $k$-prime divisor. If $D_i$ is geometrically integral,
$I_i(v)$ has one element, and we abbreviate the corresponding
intersection multiplicity to $n_v(\sD_i,x)$.
Campana weak approximation is the density of $\CC_S(X,D)$ in
$\prod_v\CC_v(X,D)$ with the product topology. We require the global
approximating point to lie outside the boundary, although local
points on the boundary are allowed. Only finitely many local
approximations are prescribed, but the global point must satisfy
the inequalities at every finite $v\notin S$.

\begin{lem}\label{lem:open}
The space $\CC_v(X,D)$ is open in $X(k_v)$, and
$X^\circ(k_v)\cap\CC_v(X,D)$ is dense in $\CC_v(X,D)$.
\end{lem}

\begin{proof}
At a finite place, reduction of sections modulo a fixed power of
$\mathfrak m_v$ is locally constant for the $v$-adic topology.
Near the reduction of a section, choose finitely many generators
of the ideal of each $\sD_{i,\lambda}$. Its intersection multiplicity
is the minimum of their valuations. A finite minimum is locally
constant; near a common zero all these valuations exceed any fixed
bound. There are only finitely many local components, and all
weights are finite, so the Campana conditions are open. The assertion
is immediate at places of $S$. Finally, a proper closed subset of a smooth geometrically integral
$k_v$-variety contains no nonempty $v$-adic open subset.
Applying this to $\Supp D$ proves the density assertion.
\end{proof}

\smallskip
\noindent\textbf{Avoiding a codimension-two subset.}
We use the standard fact that $\PP^n_k\setminus Z$ satisfies strong
approximation whenever $Z\subset\PP^n_k$ is closed of codimension
at least two; see, for example, \cite[Proposition~2.7]{ChenAP}.
Let $\sZ$ be a closed model of $Z$ and put
$\mathcal W=\PP^n_{\OO_S}\setminus\sZ$. For any finite
$T\supseteq S$ such that $\mathcal W(\OO_v)\ne\varnothing$ for
every $v\notin T$, and any nonempty open sets
$V_v\subset(\PP^n\setminus Z)(k_v)$, $v\in T$, strong approximation
gives a rational point $a$ belonging to every $V_v$ and satisfying
\[
 a\in\mathcal W(\OO_v)\qquad(v\notin T).
\]
Thus the reduction of $a$ avoids $\sZ$ outside $T$; no such
reduction condition is imposed at the prescribed approximation places.

We shall use the following arithmetic of intersections of two quadrics.

\begin{prop}\label{prop:pointed-quadrics-arithmetic}
Let $Y\subset\PP^{m+2}_k$ be a smooth geometrically integral
intersection of two quadrics, with $m\geq2$ and $P_0\in Y(k)$.
 Then
$\overline{Y(k)}=Y(\A_k)^{\Br(Y)}$, where closure is taken in
the usual adelic topology (the product topology since $Y$ is proper).
If $m=2$, the group $\Br(Y)/\Br(k)$ is finite and killed by $2$. If $m\geq3$, then
$\Br(Y)=\Br(k)$ and $Y$ satisfies ordinary weak approximation.
\end{prop}

\begin{proof}
If $m=2$, then $Y$ is a del Pezzo surface of degree $4$.
The density statement is \cite[Theorem~0.1]{SS}, and
$\Br(Y)/\Br(k)$ is isomorphic to $0$, $\ZZ/2$ or $(\ZZ/2)^2$
by \cite[Remark~1.2(1)]{VAV}. 

For $m=3$, the Brauer--Manin approximation statement follows from
\cite[Theorem~1]{CTS92}; for $m\geq4$, ordinary weak approximation
is just \cite[Theorem~3.11]{CTSSD}. We explain why the former statement
also gives ordinary weak approximation. 

For every $m\geq3$,
the Lefschetz hyperplane theorem gives $\Pic(Y_{\kbar})=\ZZ[\OO_Y(1)]$ and
$H^2_{\et}(Y_{\kbar},\mu_r)=\ZZ/r$, generated by the hyperplane
class. In the Kummer sequence
\[
 0\longrightarrow\Pic(Y_{\kbar})/r
 \longrightarrow H^2_{\et}(Y_{\kbar},\mu_r)
 \longrightarrow\Br(Y_{\kbar})[r]\longrightarrow0,
\]
the first map is therefore an isomorphism. Thus
$\Br(Y_{\kbar})=0$. The hyperplane class is defined over $k$, so
Hochschild--Serre gives
\[
 \Br(Y)/\Br(k)\hookrightarrow
 H^1(k,\Pic(Y_{\kbar}))=H^1(k,\ZZ)=0.
\]
Hence
$\Br(Y)=\Br(k)$, and global reciprocity gives
$Y(\A_k)^{\Br(Y)}=Y(\A_k)$.
\end{proof}

The rational point in the proposition is part of its hypotheses;
it will be supplied by the construction of each cover. In the surface
case, the result of Salberger--Skorobogatov is being applied only
after this point is available. 
The distinction between surfaces and higher-dimensional varieties
explains the extra discriminant argument in Section~\ref{sec:hyperplanes}
and the use of the following residue lemma in the two-conic case.

\begin{lem}\label{lem:residue}
Let $R$ be the valuation ring of a nonarchimedean local field $K$
of odd residue characteristic, with residue field $\FF$. Let
$\mathcal M$ be smooth, separated and of finite type over $R$.
If $P,Q\in\mathcal M(R)$ reduce to a locally closed subvariety
$A\simeq\Aff^r_{\FF}$ of $\mathcal M_{\FF}$, then
$\beta(P)=\beta(Q)$ for every $\beta\in\Br(\mathcal M_K)[2]$.
\end{lem}

\begin{proof}

By \cite[Corollary~2.5 and Proposition~5.1]{Bright},  we have
\[
 \partial:\Br(\mathcal M_K)[2]\longrightarrow
 H^1(\mathcal M_{\FF},\ZZ/2),\qquad
 \partial_K(\beta(P))=\bar P^*(\partial\beta).
\]
The smoothness of $\mathcal M$ over $R$ makes its special fibre a regular
divisor; if it is disconnected, the residue map is taken on each
component. The local residue map is an isomorphism
$\partial_K:\Br(K)[2]\xrightarrow{\sim}H^1(\FF,\ZZ/2)$,
since $R$ is henselian and $\Br(\FF)=0$ for a finite field.
By Kummer theory,
$H^1(\Aff^r_{\FF},\ZZ/2)=H^1(\FF,\ZZ/2)$, since the only
units of $\FF[t_1,\ldots,t_r]$ are constants and its Picard group
is zero. Thus $(\partial\beta)|_A$ is constant, and its evaluations
at $\bar P$ and $\bar Q$ agree. Injectivity of $\partial_K$ then
gives $\beta(P)=\beta(Q)$ in $\Br(K)$.
\end{proof}

The reductions $\bar P$ and $\bar Q$ need not be the same point.
The lemma compares their evaluations through the residue class on
$A$, rather than through local constancy in a neighbourhood of one
$K$-point. It does not require $\beta$ to extend to $\Br(\mathcal M)$:
its possible ramification along the special fibre is precisely
what the residue map records.

\section{The projective line}\label{sec:line-proof}

\subsection{Geometric signatures}
Assume the hypotheses of Theorem~\ref{thm:line} and $D\ne0$. Let $r$ be the number of points of
$\Supp(D_{\kbar})$.
Each geometric boundary point contributes at least $1/2$, so \eqref{eq:log-fano-degree} implies $r\leq3$. If there are
three, their ordered multiplicities $a\leq b\leq c$ satisfy
$1/a+1/b+1/c>1$. Thus $a=2$, and either $b=2$, or $b=3$ and
$c<6$. The possible signatures are
\begin{equation}\label{eq:spherical-signatures}
 (2,2,m)\ (m\geq2),\qquad (2,3,3),\qquad (2,3,4),\qquad (2,3,5).
\end{equation}

To include $r=1$ and $r=2$ in a single construction, choose an
auxiliary divisor $\Delta\geq D$ as follows. If $r=1$, the point is rational; add another
rational point of the same multiplicity. If $r=2$, replace both multiplicities by their maximum $N$. If $r=3$, take $\Delta=D$. These operations are Galois invariant and give
\begin{equation}\label{eq:auxiliary-divisor}
	\Delta=\sum_{j\in J}\left(1-\frac1{e_j}\right)Q_j,
	\qquad Z=\Supp\Delta,\qquad U_\Delta=\PP^1\setminus Z.
\end{equation}
Every original $P_i$ is a $Q_j$ with $e_j\geq n_i$. The geometric
signature of $\Delta$ is $(N,N)$ or one of
\eqref{eq:spherical-signatures}.

The auxiliary divisor is necessary for the uniformization used below.
An orbifold with one nontrivial point, or two points of unequal
multiplicities, does not have a scheme as its universal orbifold
cover; see \cite[Propositions~5.5 and~5.7]{BN}. We uniformize
$(\PP^1,\Delta)$, not the original pair in these cases. This neither
replaces the local conditions for $D$ by those for $\Delta$ nor requires
$n_i$ to divide $e_j$.

\subsection{Pointed descent and the deck group}
The spherical signatures in \eqref{eq:auxiliary-divisor} admit
universal orbifold covers over $\PP^1_{\kbar}$; see
Behrend--Noohi~\cite[Proposition~5.5]{BN}. Their geometric deck groups
are recorded below. 
\begin{center}
\begin{tabular}{@{}lll@{}}
\toprule
Signature of $\Delta_{\kbar}$ & $H$ & $\deg\bar\pi=|H|$\\
\midrule
$(N,N)$ & $\ZZ/N\ZZ$ & $N$\\
$(2,2,m)$ & $D_{m}$ & $2m$\\
$(2,3,3)$ & $A_4$ & $12$\\
$(2,3,4)$ & $S_4$ & $24$\\
$(2,3,5)$ & $A_5$ & $60$\\
\bottomrule
\end{tabular}
\end{center}
Here $D_{m}$ has presentation
$\langle r,s:r^m=s^2=1,\ srs=r^{-1}\rangle$; in particular,
$D_2\simeq(\ZZ/2)^2$. For geometric multiplicities
$\varepsilon_1,\ldots,\varepsilon_r$, the orbifold presentation is
\[
 H\simeq\langle\gamma_1,\ldots,\gamma_r:
 \gamma_i^{\varepsilon_i}=1,\ \gamma_1\cdots\gamma_r=1\rangle.
\]
The stabilizer of a geometric point above $Q_j$ is cyclic of order
$e_j$. The action is free precisely away from the ramification divisor.
The table specifies the geometric group, not a constant group scheme
over $k$. 

\begin{prop}\label{prop:pointed-cover}
There is a geometrically Galois finite morphism
$\pi:Y\to\PP^1_k$, with $Y\simeq\PP^1_k$, branched precisely over
$Z$. Its deck group scheme
\[
 G=\underline{\Aut}_{\PP^1_k}(Y)
\]
is finite \'etale over $k$, with $G(\kbar)=H$. If
$R_j=(\pi^{-1}(Q_j))_{\red}$ and $Y_U=\pi^{-1}(U_\Delta)$, then
\begin{equation}\label{eq:pullback-boundary}
 \pi^*Q_j=e_jR_j,
\end{equation}
and $Y_U\to U_\Delta$ is a left $G$-torsor.
\end{prop}

\begin{proof}
Choose $b\in U_\Delta(k)$ and $y_b\in\bar\pi^{-1}(b)(\kbar)$.
Write $\Gamma_k=\Gal(\kbar/k)$. For $\sigma\in\Gamma_k$, the
conjugate cover is again universal for the same orbifold, since
$\Delta$ is defined over $k$. Uniqueness of the universal cover,
and the simply transitive action of $H$ on the fibre above $b$,
give a unique isomorphism of pointed covers
\[
 \phi_\sigma:{}^\sigma(\bar Y,\bar\pi,y_b)
       \xrightarrow{\ \sim\ }(\bar Y,\bar\pi,y_b).
\]
Consequently $\phi_{\sigma\tau}=\phi_\sigma\circ{}^\sigma\phi_\tau$.
The cover and the point are defined over a finite extension. Taking
a finite Galois extension containing their fields of definition,
uniqueness makes these maps a finite descent datum. Effective descent
yields $\pi:Y\to\PP^1_k$ and $y_b\in Y(k)$. Since $Y$ is smooth
projective of geometric genus zero, $Y\simeq\PP^1_k$.

The induced action on $H$ is
\begin{equation}\label{eq:deck-galois-action}
 \rho_\sigma(h)=\phi_\sigma\circ{}^\sigma h\circ\phi_\sigma^{-1}
 \qquad(h\in H).
\end{equation}
The cocycle identity for $\phi$ shows that $\rho$ is a continuous
$\Gamma_k$-action by group automorphisms. Finite groups with such
action correspond to finite \'etale $k$-group schemes
\cite[Tag~04JI]{Stacks}. Let $G$ correspond to $(H,\rho)$. The
action on $\bar Y$ descends with it. Equivalently, for a $k$-scheme
$T$, one has $G(T)=\Aut_{\PP^1_T}(Y_T)$; this identifies $G$ with
the group scheme in the statement.

Over $U_\Delta$ the cover is finite \'etale. After extension to
$\kbar$, $H$ acts simply transitively on each geometric fibre.
Thus the natural morphism
\begin{equation}\label{eq:deck-torsor}
 G\times_kY_U\longrightarrow Y_U\times_{U_\Delta}Y_U,
       \qquad(g,y)\longmapsto(g\cdot y,y),
\end{equation}
is an isomorphism, which proves the torsor assertion. Finally, the
geometric ramification indices give \eqref{eq:pullback-boundary}.
\end{proof}

In particular, $G(k)=H^{\Gamma_k}$ is the group of deck
transformations defined over $k$, and need not have $|H|$ elements.
The group scheme $G$ is constant exactly when $\rho$ is trivial.
Geometric Galoisness therefore does not assert that $k(Y)/k(t)$ is
Galois.

\subsection{Twists and local specialization}
Let $\sigma_Y$ denote the semilinear Galois action on $Y_{\kbar}$
coming from $Y/k$; then $\rho_\sigma(h)=\sigma_Yh\sigma_Y^{-1}$.
We fix the twisting convention. A continuous $1$-cocycle for $G$
is a family $c_\sigma\in H$ satisfying
\begin{equation}\label{eq:deck-cocycle}
 c_{\sigma\tau}=c_\sigma\rho_\sigma(c_\tau).
\end{equation}
It defines the twist $Y_c$ by the semilinear action
\begin{equation}\label{eq:twisted-action}
 \sigma\star y=c_\sigma\cdot\sigma_Y(y).
\end{equation}
As deck transformations fix the target, the map descends to
$\pi_c:Y_c\to\PP^1_k$. The descended deck group is the inner form
$G_c$, with Galois action on $H$ given by
\begin{equation}\label{eq:twisted-deck-group}
 \rho^c_\sigma(h)=c_\sigma\rho_\sigma(h)c_\sigma^{-1}.
\end{equation}
Thus $(Y_c)_{U_\Delta}\to U_\Delta$ is a $G_c$-torsor, not in
general a torsor under the unchanged $G$; see
\cite[\S2.2]{Santens}. If $G$ is commutative, this inner form is
$G$ itself. 

For $a\in U_\Delta(k)$, choose $y_a\in\pi^{-1}(a)(\kbar)$ and
let $c_\sigma$ be the unique element of $H$ with
\begin{equation}\label{eq:pointed-twist-cocycle}
 c_\sigma\cdot\sigma_Y(y_a)=y_a.
\end{equation}
Freeness on the fibre gives \eqref{eq:deck-cocycle}. The resulting
pointed twist is
\begin{equation}\label{eq:pointed-twist}
 \pi_a:Y_a\to\PP^1_k,\qquad y_a\in Y_a(k),\qquad\pi_a(y_a)=a.
\end{equation}
Since $Y_a$ has geometric genus zero and a rational point,
$Y_a\simeq\PP^1_k$.

Put $R=\bigcup_jR_j$. Let $T\supseteq S$ be a finite set
containing the places dividing $|H|$ such that over $\OO_T$ the
given base model is $\PP^1_{\OO_T}$ and the following data extend:
$\pi$ to a finite morphism with smooth proper model $\sY$;
$G$ to a finite \'etale group scheme $\mathcal G$ acting on $\sY$
over the base; and $Z,R$ to finite \'etale divisors $\sZ,\sR$
with distinct geometric points in each fibre. Require also
\begin{equation}\label{eq:integral-pullback}
 \pi^*\sQ_j=e_j\sR_j.
\end{equation}
These are spreading-out conditions on the original cover, chosen
before any $a$ or twist; $S$ and the local Campana spaces remain fixed.
In particular, geometric deck transformations extend over the integral
closure of each valuation ring at $v\notin T$.

For $v\notin T$, the irreducible components $\sQ_{j,\lambda}$
of $\sZ_{\OO_v}$ are pairwise disjoint and finite \'etale over
$\OO_v$. A section of $\PP^1_{\OO_v}$ can meet at most one of them.
If it meets $\sQ_{j,\lambda}$, its reduction supplies an
$\FF_v$-point of this component. Since $\OO_v$ is henselian,
a connected finite \'etale $\OO_v$-scheme with such a point has
degree one; hence $\sQ_{j,\lambda}\simeq\Spec\OO_v$.
Near that reduction, $\sQ_j$ and $\sQ_{j,\lambda}$ have the same
ideal. Consequently, for every local section, at most one local
boundary intersection multiplicity is positive, and in that case
\begin{equation}\label{eq:line-local-component}
 n_v(\sQ_j,x)=n_v(\sQ_{j,\lambda},x),\qquad
 n_v(\sQ_{h,\mu},x)=0\quad((h,\mu)\ne(j,\lambda)).
\end{equation}
Thus outside $T$ the calculation with $\sQ_j$ also verifies each
of the local-component inequalities in \cite[Definition~2.4]{NS}.
At $T\setminus S$, the original Campana conditions will instead be
preserved by local approximation.

\begin{lem}\label{lem:line-local}
Let $v\notin T$, put $K=k_v$, and let
$\pi_c:Y_c\to\PP^1_K$ be any twist of $\pi_K$. If $q\in Y_c(K)$
and $\pi_c(q)$ reduces into $\sZ$, then there is $r\in Y_c(K)$
with $\pi_c(r)\in Z(K)$. Consequently, if $Y_c(K)\ne\varnothing$,
there is a nonempty open $W\subset Y_c(K)$ such that
$\pi_c(W)\subset\CC_v(\PP^1,D)\cap U_\Delta(K)$.
\end{lem}

\begin{proof}
Identify $(Y_c)_{\overline K}$ with $Y_{\overline K}$ using its
cocycle $c$. The point corresponding to $q$ satisfies
$c_\sigma\sigma_Y(y)=y$. Since $K$ is henselian, its valuation
extends uniquely to $\overline K$. The point $y$ is defined over a
finite extension, so properness of the original model gives a
reduction $\bar y$. Its image is the reduction of $\pi_c(q)$,
which lies in $\sZ$; hence $\bar y$ lies in $\sR$ by
\eqref{eq:integral-pullback}.

Finite \'etaleness over the henselian valuation ring gives an
equivariant reduction bijection \cite[Lemma 10.153.7, Tag 04GK]{Stacks}
\begin{equation}\label{eq:etale-reduction}
 R(\overline K)\xrightarrow{\sim}\sR(\overline{\FF}_v).
\end{equation}
This bijection concerns only the finite ramification divisor.
It commutes with Galois and the extended deck action, so reduction
of $c_\sigma\sigma_Y(y)=y$ fixes $\bar y$ under the twisted action. Hence the
unique lift of $\bar y$ is fixed by every $c_\sigma\sigma_Y$ and
descends to $r\in Y_c(K)$ with $\pi_c(r)\in Z(K)$. 

For the last assertion, if the image of a local point reduces
outside $\sZ$, all local boundary intersection multiplicities are zero,
and a sufficiently small neighbourhood gives the required $W$.
Otherwise use $r$ above and put $p=\pi_c(r)\in Q_j(K)$.
The closure $\sQ_{j,\lambda}$ of $p$ is the degree-one component
in \eqref{eq:line-local-component}, even when $Q_j$ is not rational
over $k$. Choose a local parameter $u$ at $r$ and an integral
local equation $z$ for $\sQ_{j,\lambda}$ near its reduction.
Ramification gives
\[
 z\circ\pi_c=u^{e_j}h(u),\qquad h(0)\in K^\times.
\]
For $q'$ sufficiently close to $r$ and off the ramification divisor,
\begin{equation}\label{eq:local-intersection-growth}
 n_v(\sQ_{j,\lambda},\pi_c(q'))
 =e_j\ord_v(u(q'))+\ord_v(h(0))\longrightarrow+\infty.
\end{equation}
All other local boundary intersection multiplicities are zero by
\eqref{eq:line-local-component}. The factor $h(0)$ need not be a
unit; its valuation is fixed. If $Q_j=P_i$ is an original boundary
point, choose the displayed intersection multiplicity at least $n_i$;
an auxiliary boundary point imposes no condition. Such $q'$ exist
because $Y_c$ is a smooth curve with the $K$-point $r$.
Lemma~\ref{lem:open} now gives the required nonempty open $W$.
\end{proof}


\Needspace{5\baselineskip}
\subsection{The global argument}
\begin{proof}[Proof of Theorem~\ref{thm:line}]
The case $D=0$ is ordinary weak approximation. Otherwise we use the
auxiliary divisor $\Delta$ and the cover of
Proposition~\ref{prop:pointed-cover}. Fix
$(x_v)_v\in\prod_v\CC_v(\PP^1,D)$ and finitely many approximation
conditions, keeping $S$ fixed. Choose a finite $T\supseteq S$
containing the prescribed approximation places and those dividing
$|H|$, such that over $\OO_T$ the given base model is $\PP^1_{\OO_T}$,
$\pi$ extends to a finite morphism with smooth proper model $\sY$,
$G$ extends to a finite \'etale group scheme acting on $\sY$ over
the base, and $\sZ,\sR$ are finite \'etale divisors with
$\pi^*\sQ_j=e_j\sR_j$. By Lemma~\ref{lem:open}, choose nonempty
open sets $V_v\subset U_\Delta(k_v)\cap\CC_v(\PP^1,D)$, $v\in T$,
compatible with the prescribed approximation conditions.
Choose $a\in U_\Delta(k)$ in all $V_v$. The pointed twist
$\pi_a:Y_a\to\PP^1$ has a point $q\in Y_a(k)$ above $a$, so
$Y_a\simeq\PP^1_k$.

Choose a finite $T'\supseteq T$ such that $\pi_a$ extends over
$\OO_{T'}$ to a finite morphism of smooth proper models, and the
closures $\sR_{j,a}$ of $(\pi_a^{-1}(Q_j))_{\red}$ satisfy
\begin{equation}\label{eq:twist-integral-pullback}
 \pi_a^*\sQ_j=e_j\sR_{j,a}
\end{equation}
as effective Cartier divisors. At $v\in T$ take a neighbourhood
of $q$ mapping into $V_v$. At $v\in T'\setminus T$,
Lemma~\ref{lem:line-local} (for $q$) gives a nonempty open
set mapping to $\CC_v(\PP^1,D)\cap U_\Delta(k_v)$.
Weak approximation on $Y_a\simeq\PP^1_k$ gives a rational point
$p$ in all these open sets, with $t=\pi_a(p)\in U_\Delta(k)$.

The prescribed approximation conditions and the Campana conditions
at $T'\setminus S$ hold by construction. For $v\notin T'$, properness
and \eqref{eq:twist-integral-pullback} give
\[
 n_v(\sQ_j,t)=e_j n_v(\sR_{j,a},p)\in e_j\ZZ_{\geq0}.
\]
By \eqref{eq:line-local-component}, this is the unique possibly
positive local-component intersection multiplicity over $Q_j$.
If $Q_j=P_i$, every positive value is at least $e_j\geq n_i$.
All other local-component intersection multiplicities vanish. Thus
$t\in\CC_S(\PP^1,D)$ gives the required approximation.
\end{proof}

\section{Hyperplane arrangements}\label{sec:hyperplanes}

Let $N=n+3$ and let $E_i=\{\ell_i=0\}$ be as in
Theorem~\ref{thm:hyperplanes}. Put $U=\PP^n\setminus\bigcup_iE_i$.
Let $A$ be the $N\times(n+1)$ matrix of the forms $\ell_i$, and
choose a basis of their two relations:
\begin{equation}\label{eq:hyperplane-relations}
 \sum_i b_i\ell_i=\sum_i c_i\ell_i=0,\qquad
 B=\begin{pmatrix}b_1&\cdots&b_N\\c_1&\cdots&c_N\end{pmatrix}.
\end{equation}
General position is equivalent to
$\Delta_{ij}:=b_ic_j-c_ib_j\ne0$ for $i\ne j$. 

Let $T\supseteq S$ be a finite set containing the places above
$2$ such that outside $T$ the matrices $A$ and $B$ are integral
with unit maximal minors and the forms define the prescribed
closures $\sE_i$. The
unit-minor condition gives
$\ker(B:\OO_v^{N}\to\OO_v^2)=A\OO_v^{n+1}$, with $AX$ primitive
if and only if $X$ is primitive. Each $E_i$ is geometrically integral, so it has just one local
irreducible component at every place. For $v\notin T$ and primitive
$X\in\OO_v^{n+1}$ off the boundary, the pulled-back ideal is
generated by $\ell_i(X)$; hence its intersection multiplicity in the
sense of \cite[Definitions~2.3--2.4]{NS} is
$n_v(\sE_i,[X])=\ord_v(\ell_i(X))$.

Put
\[
 Z=\bigcup_{i<j}(E_i\cap E_j),\qquad
 \sZ=\bigcup_{i<j}(\sE_i\cap\sE_j),\qquad
 \mathcal W=\PP^n_{\OO_S}\setminus\sZ.
\]
Require also that $\mathcal W(\OO_v)\ne\varnothing$ for every
$v\notin T$. For $a\in U(k)$ and $v\notin T$, the condition
$a\in\mathcal W(\OO_v)$ says precisely that, after choosing a
primitive local representative $\widetilde a$, at most one value
$\ell_i(\widetilde a)$ has positive valuation. Such $a$ exists by strong approximation on $\PP^n\setminus Z$.

For $a\in U(k)$, choose a homogeneous representative
$\widetilde a$ and put $a_i=\ell_i(\widetilde a)$. Define
\begin{equation}\label{eq:hyperplane-cover}
 Y_a:\quad Q_b=\sum_i b_i a_i y_i^2=0,\qquad
 Q_c=\sum_i c_i a_i y_i^2=0\quad\subset\PP^{n+2}_k.
\end{equation}
The linear embedding $\iota(x)=[\ell_1(x):\cdots:\ell_N(x)]$
identifies $\PP^n$ with the subspace defined by \eqref{eq:hyperplane-relations}.

\begin{prop}\label{prop:hyperplane-cover}
The variety $Y_a$ is a smooth geometrically integral intersection
of two quadrics of dimension $n$. There is a finite morphism of
degree $2^{n+2}$,
\[
 \pi_a:Y_a\longrightarrow\PP^n_k,\qquad
 \iota(\pi_a(y))=[a_1y_1^2:\cdots:a_Ny_N^2],
\]
with $q_0=(1:\cdots:1)\in Y_a(k)$ and $\pi_a(q_0)=a$.
The morphism $\pi_a$ does not depend on the choice of homogeneous representative
$\widetilde a$.
\end{prop}

\begin{proof}
Let $M_b=\operatorname{diag}(b_ia_i)_{i=1}^N$ and
$M_c=\operatorname{diag}(c_ia_i)_{i=1}^N$. The discriminant of
the pencil $sQ_b+tQ_c$ is
\[
 \det(sM_b+tM_c)
 =\left(\prod_{i=1}^N a_i\right)
   \prod_{i=1}^N(sb_i+tc_i).
\]
Since $a\in U(k)$, all $a_i$ are nonzero. Moreover,
$\Delta_{ij}\ne0$ for $i\ne j$, so the linear factors
$sb_i+tc_i$ are nonzero and pairwise nonproportional. Thus the
discriminant has $N$ distinct zeros in $\PP^1_{\kbar}$.
By Reid's criterion \cite[Proposition~2.1]{Reid72}, applied over
$\kbar$, $Y_a$ is a smooth complete intersection of two quadrics
of dimension $N-3=n$.

The finite morphism
\[
 F_a:\PP^{N-1}\longrightarrow\PP^{N-1},\qquad
 [y_1:\cdots:y_N]\longmapsto[a_1y_1^2:\cdots:a_Ny_N^2]
\]
has $Y_a=F_a^{-1}(\iota(\PP^n))$, and hence induces the stated
finite morphism $\pi_a$. The generic fibre of $(\pi_a)_{\kbar}$
is obtained over $\kbar(\PP^n)$ by adjoining square roots of
\[
 f_i=\frac{a_N\ell_i}{a_i\ell_N},\qquad 1\leq i<N.
\]
These squareclasses are independent, since
$\ord_{E_j}(f_i)=\delta_{ij}$ for $1\leq i,j<N$. Thus the
generic fibre of $(\pi_a)_{\kbar}$ is integral of degree
$2^{N-1}$.
Every geometric irreducible component of $Y_a$ has dimension
$n$ and therefore dominates $\PP^n$ under the finite morphism
$\pi_a$. Together with smoothness, this proves that $Y_a$ is
geometrically integral and that $\deg(\pi_a)=2^{N-1}=2^{n+2}$.

The relations \eqref{eq:hyperplane-relations} give
$q_0=(1:\cdots:1)\in Y_a(k)$ and $\pi_a(q_0)=a$.
Replacing $\widetilde a$ by $\lambda\widetilde a$ multiplies all
$a_i$ by the same $\lambda\in k^\times$, leaving both $Y_a$
and $F_a$ unchanged. Thus $\pi_a$ is independent of this choice.
\end{proof}

For $n\geq3$, Proposition~\ref{prop:pointed-quadrics-arithmetic}
gives ordinary weak approximation on every $Y_a$. For $n=2$, we
make a suitable choice of $a$.

\begin{prop}\label{prop:hyperplane-choice}
Suppose $n=2$. For any nonempty open sets
$V_v\subset U(k_v)$, $v\in T$, there exists $a\in U(k)$
belonging to all $V_v$ such that $\Br(Y_a)=\Br(k)$ and
$a\in\mathcal W(\OO_v)$ for every $v\notin T$. This $Y_a$
satisfies ordinary weak approximation.
\end{prop}

\begin{proof}
The five singular members of the pencil are defined over $k$. The member
$c_iQ_b-b_iQ_c=-\sum_{h\ne i}\Delta_{ih}a_hy_h^2$ has rank four
and discriminant squareclass
\begin{equation}\label{eq:hyperplane-discriminants}
 \delta_i=d_i\prod_{h\ne i}a_h\in k^\times/k^{\times2},\qquad
 d_i=\prod_{h\ne i}\Delta_{ih},\qquad \prod_{i=1}^5\delta_i=1.
\end{equation}
The product relation follows because each $a_h$ occurs four times and
$\prod_i d_i=\prod_{i<h}\Delta_{ih}^2$.
By \cite[Theorem~3.37 and Example~3.41]{Wit07}, it suffices to choose
$a$ so that $\delta_1,\ldots,\delta_4$ are independent: then all five
classes are nontrivial, and their product is their only nontrivial relation.

Choose four distinct finite places $w_1,\ldots,w_4\notin T$ with
sufficiently large residue fields. At $w_j$, impose, using a primitive
local representative of $a$,
\begin{equation}\label{eq:hyperplane-auxiliary-intersections}
 \ord_{w_j}(a_j)=1,\qquad \ord_{w_j}(a_i)=0\quad(i\ne j).
\end{equation}
This defines a nonempty open subset of $U(k_{w_j})$: lift a residue
point of $E_j$ lying on none of the other lines, with intersection multiplicity exactly one along $\sE_j$.
The auxiliary open sets lie in $\mathcal W(\OO_{w_j})$.
Strong approximation on $\PP^2\setminus Z$, applied to these open
sets and the $V_v$, gives $a\in U(k)$ satisfying all the conditions
and integral in $\mathcal W$ outside $T\cup\{w_1,\ldots,w_4\}$.
The prescribed intersection multiplicities at the $w_j$ show that in fact
$a\in\mathcal W(\OO_v)$ for every $v\notin T$.

Since each $d_i$ is a unit at the $w_j$, the valuation parities of
$\delta_1,\ldots,\delta_4$ at these places form the matrix
\[
 M=\begin{pmatrix}
 0&1&1&1\\1&0&1&1\\1&1&0&1\\1&1&1&0
 \end{pmatrix},\qquad M^2=I_4\quad\text{over }\FF_2.
\]
These parities are independent of the choice of homogeneous representative:
rescaling multiplies each discriminant representative by a fourth power.
Thus a square relation $\prod_{i=1}^4\delta_i^{\epsilon_i}=1$
gives $M\epsilon=0$, hence $\epsilon=0$.
Wittenberg's criterion yields $\Br(Y_a)/\Br(k)=0$.
Ordinary weak approximation now follows from
Proposition~\ref{prop:pointed-quadrics-arithmetic}.
\end{proof}

\begin{lem}\label{lem:hyperplane-repair}
Let $a\in U(k)$ and let $v\notin T$ be finite. Suppose that the
reduction of $a$ lies outside $\sZ$. Then there exists
$y_v\in Y_a(k_v)$ with
$\pi_a(y_v)\in U(k_v)\cap\CC_v(\PP^n,D)$.
\end{lem}

\begin{proof}
Put $R=\OO_v$ and $\varpi=\varpi_v$, and rescale
$\widetilde a$ locally to be integral and primitive. If all
$a_i=\ell_i(\widetilde a)$ are units, take $y_v=q_0$.
Otherwise exactly one, say $a_j$, has positive valuation.
Since $\ell_j$ is primitive over $R$, choose $d\in R^{n+1}$ with
$\ell_j(d)=1$ and set
\[
 X=\widetilde a+(\varpi^2-1)a_jd.
\]
Then $X\equiv\widetilde a\pmod\varpi$ is primitive,
$\ell_j(X)=\varpi^2a_j$, and
$\ell_i(X)/a_i\in1+\varpi R$ for $i\ne j$.
These latter ratios are squares by Hensel's lemma. Set
$y_j=\varpi$ and choose $y_i\in1+\varpi R$ with
$y_i^2=\ell_i(X)/a_i$ for $i\ne j$.
The relations among the $\ell_i$ give $y_v\in Y_a(k_v)$ and
$\pi_a(y_v)=[X]$. Consequently
\begin{equation}\label{eq:hyperplane-repaired-intersections}
 n_v(\sE_i,[X])=
 \begin{cases}
 \ord_v(a_j)+2\geq3,&i=j,\\
 0,&i\ne j.
 \end{cases}
\end{equation}
All $\ell_i(X)$ are nonzero, so $[X]\in U(k_v)$, as required.
\end{proof}

\begin{proof}[Proof of Theorem~\ref{thm:hyperplanes}]
Fix $(x_v)_v\in\prod_v\CC_v(\PP^n,D)$ and finitely many approximation
conditions, keeping $S$ fixed. Choose a finite $T\supseteq S$
containing the prescribed approximation places and those above $2$,
such that, for every $v\notin T$, the matrices $A,B$ are integral
with unit maximal minors, the forms $\ell_i$ define the prescribed
closures $\sE_i$, and $\mathcal W(\OO_v)\ne\varnothing$.
By Lemma~\ref{lem:open}, choose nonempty open sets
$V_v\subset U(k_v)\cap\CC_v(\PP^n,D)$, $v\in T$, compatible with
the prescribed approximation conditions. For $n\geq3$, strong
approximation on $\PP^n\setminus Z$ gives $a\in U(k)$ in all $V_v$
with $a\in\mathcal W(\OO_v)$ for every $v\notin T$.
For $n=2$, use Proposition~\ref{prop:hyperplane-choice}, which
also ensures $\Br(Y_a)=\Br(k)$. In both cases, $Y_a$ satisfies
ordinary weak approximation.

Let $T'\supseteq T$ be finite and contain every place where a chosen
global coefficient $a_i$ is not a unit. At $v\in T$, use a
neighbourhood of $q_0$ mapping into $V_v$. At $v\in T'\setminus T$,
Lemma~\ref{lem:hyperplane-repair} and Lemma~\ref{lem:open} give a
nonempty open set whose image lies in $U(k_v)\cap\CC_v(\PP^n,D)$.
This includes the four auxiliary places when $n=2$. Weak
approximation on $Y_a$ gives $y\in Y_a(k)$ in all these open sets;
we may require every $y_i$ to be nonzero.

Put $x=\pi_a(y)$. At $v\notin T'$, take primitive integral
coordinates for $y$. As all $a_i$ are units, $w=(a_i y_i^2)_i$
is primitive and belongs to $\ker B$, so $w=AX$ with $X$ integral
and primitive. Therefore
\begin{equation}\label{eq:hyperplane-good-intersection-parity}
 n_v(\sE_i,x)=\ord_v(a_i y_i^2)
             =2\ord_v(y_i)\in2\ZZ_{\geq0}.
\end{equation}
These are precisely the local-component intersection multiplicities.
The Campana conditions at $T\setminus S$ hold because $x\in V_v$,
and those at $T'\setminus T$ hold by local repair. Thus
$x\in\CC_S(\PP^n,D)$ gives the required approximation.
\end{proof}

\section{Two transverse quadrics}\label{sec:quadric-proof}

Let $C_1,C_2$ be as in Theorem~\ref{thm:quadrics}, with equations
$q_i=0$, and let $U=\PP^n\setminus(C_1\cup C_2)$. For
$a\in U(k)$ choose $\widetilde a=(a_0,\ldots,a_n)$, put
$c_i=q_i(\widetilde a)$, and define
\begin{equation}\label{eq:quadric-cover}
 Y_a:\quad q_1(X_0,\ldots,X_n)=c_1Z_1^2,\qquad
 q_2(X_0,\ldots,X_n)=c_2Z_2^2\quad\subset\PP^{n+2}_k.
\end{equation}
The projection gives $\pi_a:Y_a\to\PP^n_k$ with the rational point
$P_0=[a_0:\cdots:a_n:1:1]$ above $a$.

\begin{prop}\label{prop:quadric-cover}
The variety $Y_a$ is a smooth geometrically integral intersection
of two quadrics of dimension $n$, and $\pi_a$ is finite flat of
degree four.
\end{prop}

\begin{proof}
The $X$-coordinates cannot all vanish. On each affine chart of the
base, adjoining the two roots gives a free algebra with basis
$1,z_1,z_2,z_1z_2$. Hence $\pi_a$ is finite flat of degree 4.
For any linear form $L$, the squareclasses of $q_i/L^2$ over
$\kbar(\PP^n)$ are independent; this gives geometric integrality.
A relation between the gradients has the form
\[
 \lambda(\nabla q_1,-2c_1Z_1,0)
   +\mu(\nabla q_2,0,-2c_2Z_2)=0.
\]
If both $Z_i$ are nonzero, then $\lambda=\mu=0$. If precisely
one vanishes, smoothness of the corresponding quadric gives the
same conclusion. If both vanish, transversality gives independence
of the gradients. Thus $Y_a$ is smooth.
\end{proof}

In dimension two, transversality says that the conics meet in four
distinct geometric points. 
Adjunction
gives $-K_{Y_a}=(n-1)\OO_{Y_a}(1)$, so the surface cover is a
del Pezzo surface of degree 4. The degree of the map $\pi_a$
is 4 in every dimension, whereas the hyperplane cover has degree
$2^{n+2}$; the covers in both cases are smooth complete intersections of two quadrics in their
projective embeddings.

Let $T\supseteq S$ be a finite set containing the places above
$2$ such that outside $T$ the $q_i$ are integral and define the
prescribed closures, the quadrics have smooth reduction, and their
intersection remains transverse. Since $n\geq2$, each smooth
quadric $C_i$ is geometrically integral: a reducible quadratic
form over $\kbar$ would define two hyperplanes (or a double
hyperplane), and hence a singular hypersurface. Thus $C_i$ has
just one local irreducible component at every place. For $v\notin T$
and primitive integral coordinates $X$, the pulled-back ideal of $\sC_i$ is
generated by $q_i(X)$, so
\[
 n_v(\sC_i,[X])=\ord_v(q_i(X))\qquad(v\notin T)
\]
is the intersection multiplicity of \cite[Definitions~2.3--2.4]{NS}.
Put
\[
 Z=C_1\cap C_2,\qquad \sZ=\sC_1\cap\sC_2,\qquad
 \mathcal W=\PP^n_{\OO_S}\setminus\sZ.
\]
Require also that $\mathcal W(\OO_v)\ne\varnothing$ for every
$v\notin T$. Since $Z$ has codimension two, strong
approximation on $\PP^n\setminus Z$ lets us choose $a$ whose
reduction belongs to at most one boundary quadric outside the
finite approximation set.

\begin{lem}\label{lem:quadric-repair}
Let $a\in U(k)$ and let $v\notin T$ be finite. Suppose that the
reduction of $a$ lies outside $\sZ$. Then there exists
$P_v\in Y_a(k_v)$ such that
\[
 \pi_a(P_v)\in U(k_v)\cap\CC_v(\PP^n,D),\qquad
 \alpha(P_v)=\alpha(P_0)\quad(\alpha\in\Br(Y_a)).
\]
\end{lem}

\begin{proof}
Put $R=\OO_v$ and $\varpi=\varpi_v$. Rescale $\widetilde a$
to be primitive and work in an affine chart with one unit coordinate
equal to $1$. Rescaling by $\lambda$ replaces $c_i$ by
$\lambda^2c_i$; the isomorphism
$[X:Z_1:Z_2]\mapsto[\lambda X:Z_1:Z_2]$ preserves the projection
and the distinguished point. Thus the local equations are
\begin{equation}\label{eq:local-quadric-model}
 f_1(x)=c_1z_1^2,\qquad f_2(x)=c_2z_2^2,\qquad
 c_i=f_i(a)\in R\setminus\{0\},\qquad P_0=(a,1,1).
\end{equation}
If both $c_i$ are units, take $P_v=P_0$. Otherwise, after
interchanging the indices, we may assume $\ord_v(c_1)>0$ and $c_2\in R^\times$.
Smoothness of the reduction of $C_1$ gives a coordinate $x_r$ with
$\partial\bar f_1/\partial x_r(\bar a)\ne0$. Fixing the other
base coordinates, Hensel's lemma gives $x\equiv a\pmod\varpi$
with $f_1(x)=c_1\varpi^2$. Since
$f_2(x)/c_2\in1+\varpi R$, choose $z_2\in1+\varpi R$ with
$z_2^2=f_2(x)/c_2$ and put $P_v=(x,\varpi,z_2)$. Then
\begin{equation}\label{eq:repaired-intersections}
 n_v(\sC_1,\pi_a(P_v))=\ord_v(c_1)+2\geq3,\qquad
 n_v(\sC_2,\pi_a(P_v))=0.
\end{equation}
Both quadratic values are nonzero, so the image belongs to $U(k_v)$.

For $n\geq3$, the Brauer assertion follows from
Proposition~\ref{prop:pointed-quadrics-arithmetic}. For $n=2$, let
$\mathcal M$ be the relative smooth locus of
\eqref{eq:local-quadric-model}. The line
\[
 A=\{x=\bar a,\ z_2=1\}\simeq\Aff^1_{\FF_v}
\]
lies in $\mathcal M_{\FF_v}$: along $A$, the Jacobian minor in
$x_r,z_2$ equals
$-2\bar c_2\,\partial\bar f_1/\partial x_r(\bar a)\ne0$.
Both $P_0$ and $P_v$ are sections of $\mathcal M$ reducing into
$A$, with $z_1$ respectively equal to $1$ and $0$ on reduction.
For $\alpha\in\Br(Y_a)$, the normalized class
$\beta=\alpha-\alpha(P_0)$ is $2$-torsion by
Proposition~\ref{prop:pointed-quadrics-arithmetic} and evaluation
at $P_0$. Restricting $\beta$ to $\mathcal M_{k_v}$ and applying
Lemma~\ref{lem:residue} gives
$\beta(P_v)=\beta(P_0)=0$, which proves the assertion.
\end{proof}

\begin{proof}[Proof of Theorem~\ref{thm:quadrics}]
Fix $(x_v)_v\in\prod_v\CC_v(\PP^n,D)$ and finitely many approximation
conditions, keeping $S$ fixed. Choose a finite $T\supseteq S$
containing the prescribed approximation places and those above $2$,
such that, for every $v\notin T$, the forms $q_i$ are integral and
define the prescribed closures $\sC_i$, the reductions of $C_1,C_2$
are smooth and meet transversely, and $\mathcal W(\OO_v)\ne\varnothing$.
By Lemma~\ref{lem:open}, choose nonempty open sets
$V_v\subset U(k_v)\cap\CC_v(\PP^n,D)$, $v\in T$, compatible with
the prescribed approximation conditions. Strong approximation on
$\PP^n\setminus Z$ gives $a\in U(k)$ belonging to all $V_v$ and
satisfying $a\in\mathcal W(\OO_v)$ for every $v\notin T$.
Construct $Y_a$ and $P_0$, and choose a finite $T'\supseteq T$
containing the places where either global coefficient $c_i$ is not a unit.

For $v\in T'\setminus T$, choose $P_v$ by
Lemma~\ref{lem:quadric-repair}. Put $Q_v=P_v$ at these places and
$Q_v=P_0$ elsewhere. The lemma and reciprocity give
\[
 \sum_v\inv_v\alpha(Q_v)=\sum_v\inv_v\alpha(P_0)=0
 \qquad(\alpha\in\Br(Y_a)).
\]
Here $\alpha(P_0)\in\Br(k)$ is one fixed global class, and each
local modification preserves its evaluation. In particular, we do
not assume that the surface has trivial Brauer group.
Thus $(Q_v)_v\in Y_a(\A_k)^{\Br(Y_a)}$, and
Proposition~\ref{prop:pointed-quadrics-arithmetic} gives
$P\in Y_a(k)$ sufficiently close to $Q_v$ for every $v\in T'$.
Require $\pi_a(P)\in V_v$ at $T$, and a Campana image at
$T'\setminus T$. By openness, we may also require $Z_1Z_2\ne0$.

Set $x=\pi_a(P)$. At $v\notin T'$, take primitive integral
coordinates $[X_0:\cdots:X_n:Z_1:Z_2]$ for $P$. The vector $X$
is primitive: otherwise $q_i(X)\in\mathfrak m_v^2$, and, since
$c_i$ are units, the equations would force both $Z_i$ into
$\mathfrak m_v$, contradicting primitivity of the full coordinate
vector. Hence
\begin{equation}\label{eq:good-intersection-parity}
 n_v(\sC_i,x)=\ord_v(q_i(X))
             =2\ord_v(Z_i)\in2\ZZ_{\geq0}.
\end{equation}
These are the local-component intersection multiplicities.
The Campana conditions at $T\setminus S$ hold because $x\in V_v$,
and those at $T'\setminus T$ hold by local repair. Thus
$x\in\CC_S(\PP^n,D)$ gives the required approximation.
\end{proof}

\end{document}